\documentclass[12pt, reqno]{llncs}

\usepackage{url}
\usepackage{amsmath}
\usepackage{amssymb}

\spnewtheorem*{mainconj}{Main Conjecture}{\bfseries}{\itshape}

\usepackage[usenames,dvipsnames]{pstricks}
\usepackage{epsfig}
\usepackage{pst-grad} % For gradients
\usepackage{pst-plot} % For axes

\newcommand{\seclabel}[1]{\label{sec:#1}}   % section
\newcommand{\thmlabel}[1]{\label{thm:#1}}   % theorem
\newcommand{\lemlabel}[1]{\label{lem:#1}}   % lemma

\newcommand{\secref}[1]{\ref{sec:#1}}   % section
\newcommand{\lemref}[1]{\ref{lem:#1}}   % lemma
\newcommand{\Mlt}{\mathrm{Mlt}}         % multiplication group
\newcommand{\Inn}{\mathrm{Inn}}			% inner mapping group

\newcommand{\ldiv}{\backslash}					% left division
\newcommand{\rdiv}{\slash}						% right division
\newcommand{\inv}{^{-1}}						% inverse
\newcommand{\sbl}[1]{\langle#1\rangle}			% subloop generated by #1
\newcommand{\cl}[1]{\mathrm{c\ell}(#1)}         % class
\newcommand{\Csorgo}{Cs\"{o}rg\H{o}\ }          % Csorgo

\title{Loops with abelian inner mapping groups:
an application of automated deduction\thanks{Dedicated to the memory of William McCune (1953--2011).}}

\author{Michael Kinyon\inst{1} \and
Robert Veroff\inst{2} \and
Petr Vojt\v{e}chovsk\'y\thanks{Partially supported by Simons Foundation
Collaboration Grant 210176.}\inst{1}}

\institute{Department of Mathematics, University of Denver, \\
Denver, CO 80208 USA \\
\email{mkinyon@math.du.edu}\qquad
\email{\url{www.math.du.edu/~mkinyon}} \\
\email{petr@math.du.edu}\qquad
\email{\url{www.math.du.edu/~petr}}
\and
Department of Computer Science, University of New Mexico, \\
Albuquerque, NM 87131 USA \\
\email{veroff@cs.unm.edu}\qquad
\email{\url{www.cs.unm.edu/~veroff}}
}

\begin{document}

%\keywords{Automated deduction, Prover9, loop with abelian inner mappings, LC loop, %nilpotent loop}

\maketitle

\begin{abstract}
We describe a large-scale project in applied automated deduction concerned with the following problem of considerable interest in loop theory: If $Q$ is a loop with commuting inner mappings, does it follow that $Q$ modulo its center is a group and $Q$ modulo its nucleus is an abelian group? This problem has been answered affirmatively in several varieties of loops. The solution usually involves sophisticated techniques of automated deduction, and the resulting derivations are very long, often with no higher-level human proofs available.
\end{abstract}

\section{Introduction}
\seclabel{intro}

A \emph{quasigroup} $(Q,\cdot)$ is a set $Q$ with a binary operation $\cdot$ such that for each $x\in Q$, the \emph{left translation}
$L_x : Q\to Q ; y\mapsto yL_x = xy$ and the \emph{right translation} $R_x : Q\to Q ; y\mapsto yR_x = yx$
are bijections. A quasigroup is a \emph{loop} if, in addition, there exists $1\in Q$ satisfying $1\cdot x = x\cdot 1 = x$ for all $x\in Q$.
Standard references for the theory of quasigroups and loops are \cite{Belousov}, \cite{Bruck} and \cite{Pflugfelder}.

\begin{example}\label{Ex:loop}
The above definition merely says that the multiplication table of a loop is a \emph{Latin square} (that is, every symbol occurs in every row and in every column precisely once), in which the row labels are duplicated in column $1$ and the column labels are duplicated in row $1$.

For instance, the following multiplication table defines a loop $Q$ with elements $\{1,\dots,6\}$.
\begin{displaymath}
    \begin{array}{c|cccccc}
    \cdot\enspace &\thinspace 1\thinspace &\thinspace 2\thinspace &
    \thinspace 3\thinspace &\thinspace 4\thinspace &\thinspace 5\thinspace &\thinspace 6 \\
    \hline
    1\thinspace&1&2&3&4&5&6\\
    2\thinspace&2&1&4&3&6&5\\
    3\thinspace&3&4&5&6&1&2\\
    4\thinspace&4&3&6&5&2&1\\
    5\thinspace&5&6&2&1&3&4\\
    6\thinspace&6&5&1&2&4&3
    \end{array}
\end{displaymath}
Note that $Q$ is neither commutative (as $3\cdot 5 = 1$ while $5\cdot 3 = 2$), nor associative (as $3\cdot (3\cdot 4) = 2$ while $(3\cdot 3)\cdot 4 = 1$). The left translation $L_2$ is the permutation $(1,2)(3,4)(5,6)$, the right translation $R_3$ is the permutation $(1,3,5,2,4,6)$.
\end{example}

A quasigroup can be equivalently defined as a set $Q$ with \emph{three} binary operations $\cdot$ (multiplication), $\ldiv$ (left division), and $\rdiv$ (right division) satisfying the axioms
\[
x\cdot (x\ldiv y) = y = x\ldiv (x\cdot y)\,,\qquad (x\cdot y)\rdiv y = x = (x\rdiv y)\cdot y\,.
\]
Starting with a quasigroup $(Q,\cdot)$, one merely needs to set $x\ldiv y = yL_x^{-1}$ and $x\rdiv y = xR_y^{-1}$. Conversely, starting with a three operation quasigroup \mbox{$(Q,\cdot,\ldiv,\rdiv)$}, the above axioms guarantee that all translations $L_x$, $R_x$ are bijections of $Q$. For more details on the equivalence of the two definitions, see \cite{Evans}.

Since quasigroups and loops can be equationally defined, they have been fallow ground for the techniques and tools of automated deduction. They appeared already in the milestone paper of Knuth and Bendix \cite{KB}, and interest in
them has continued in the theorem-proving community \cite{Hullot} \cite{PSATO}. In more recent years,
mathematicians specializing in quasigroups and loops have been making significant use of automated deduction tools
\cite{KepKinPhil1,diassoc,general,dicc,extra,pacc,Cloops2,commutant,AIM-LC-Web,KB,moufang,laws,rings,gloops,cc}
\cite{P-short,P-Moufang,PSh,PSt1,PSt2,PV-BM1,PV-BM2,PV-C,PV-Scoop}.
%
%\cite{KepKinPhil1}--\cite{cc},\cite{P-short}--\cite{PV-Scoop}.
%P-Moufang, PSh,PSt1,PSt2,PV-BM1,PV-BM2,PV-C,PV-Scoop}
%
With the exception of \cite{PSt1} and \cite{PSt2}, all of the aforementioned references used Bill McCune's \textsc{Prover9} \cite{Prover9} or
its predecessor \textsc{Otter}~\cite{Otter}.

The purpose of this paper is to report on progress in a large-scale project
in the application of automated deduction to loop theory. The primary tool
has been \textsc{Prover9}, and the search for proofs has relied heavily
on the method of \emph{proof sketches}~\cite{Sketches}.  Proof sketches
have been especially effective in this project, in part because of the
large number of closely related problems being considered.

In {\S}\secref{problem} we give the mathematical background and the Main Conjecture and then explain why it is a problem suitable for equational reasoning tools. In {\S}\secref{strategy} we discuss our strategy for attacking the problem with \textsc{Prover9}. As an example, in {\S}\secref{lcloops} we discuss one particular class of loops for which we were able to solve the problem. Finally in {\S}\secref{finale} we mention other classes of loops for which the solution is known.

\section{The Main Conjecture and the Project}
\seclabel{problem}

The left and right translations in a loop $Q$ generate the \emph{multiplication group} $\Mlt(Q) = \sbl{L_x, R_x \mid x\in Q}$, a subgroup of the group of all bijections on $Q$. The \emph{inner mapping group} $\Inn(Q) = \{\varphi\in\Mlt(Q)\mid 1\varphi = 1\}$ is a subgroup of $\Mlt(Q)$ consisting of all bijections in $\Mlt(Q)$ that fix the element $1$.

A loop $Q$ is said to be an \emph{AIM loop} (for \textbf{A}belian \textbf{I}nner \textbf{M}appings) if $\Inn(Q)$ is an abelian group. AIM loops are the main subject of this investigation.

A nonempty subset $S$ of a loop $Q$ is a \emph{subloop} ($S\le Q$) if it is closed under the three operations $\cdot$, $\ldiv$, $\rdiv$. A subloop $S$ of $Q$ is \emph{normal} ($S\unlhd Q$) if $S\varphi = \{s\varphi\mid s\in S\}$ is equal to $S$ for every $\varphi\in\Inn(Q)$.

To save space, we will often denote $x\cdot y$ by $xy$, and we will use $\cdot$ to indicate the priority of multiplication. For instance, $x\cdot yz$ stands for $x\cdot (y\cdot z)$.

The \emph{left}, \emph{right}, and \emph{middle nucleus} of a loop $Q$ are defined, respectively, by
\begin{align*}
N_{\lambda}(Q) &= \{ a\in Q\mid ax\cdot y = a\cdot xy,\quad \forall x,\,y\in Q\}, \\
N_{\rho}(Q) &= \{ a\in Q\mid xy\cdot a = x\cdot ya,\quad \forall x,\,y\in Q\}, \\
N_{\mu}(Q) &= \{ a\in Q\mid xa\cdot y = x\cdot ay,\quad \forall x,\,y\in Q\},
\end{align*}
and the \emph{nucleus} is $N(Q) = N_{\lambda}(Q)\cap N_{\rho}(Q)\cap N_{\mu}(Q)$. It is not hard to show that each of these nuclei is a subloop.

The \emph{commutant} of a loop $Q$ is the set
\[
C(Q) = \{ a\in Q\mid ax = xa\quad \forall x\in Q\}\,,
\]
which is not necessarily a subloop. The \emph{center} of $Q$ is
\[
Z(Q) = C(Q) \cap N(Q)\,.
\]
The center is always a normal subloop.

Thus the nucleus $N(Q)$ consists of all elements $a\in Q$ that associate with all $x$, $y\in Q$, the commutant $C(Q)$ consists of all elements $a\in Q$ that commute with all $x\in Q$, and the center $Z(Q)$ consists of all elements $a\in Q$ that commute and associate with all $x$, $y\in Q$.

For a general loop $Q$, the inclusions $N_{\lambda}(Q)\le N(Q)$, $N_{\rho}(Q)\le N(Q)$, $N_{\mu}(Q)\le N(Q)$, $N(Q)\le Z(Q)$ and $C(Q)\le Z(Q)$ hold, but not necessarily the equalities.

Given a normal subloop $S$ of a loop $Q$, denote by $Q/S$ the \emph{factor loop} $Q$ modulo $S$ whose elements are the subsets (left cosets) $xS = \{xs\mid s\in S\}$ for $x\in Q$, and where we multiply according to $xS\cdot yS = (x\cdot y)S$.

Now define $Z_0(Q)=\{1\}$, and for $i\ge 0$ let $Z_{i+1}(Q)$ be the preimage of $Z(Q/Z_i(Q))$ under the canonical projection $\pi_i:Q\to Q/Z_i(Q);\;x\mapsto xZ_i(Q)$. Note that $\{1\} = Z_0(Q)\le Z_1(Q) \le Z_2(Q) \le \cdots \le Q$. The loop $Q$ is \emph{(centrally) nilpotent of class} $n$, written $\cl{Q} = n$, if $Z_{n-1}(Q)\ne Q$ and $Z_n(Q)=Q$.

\begin{example}\label{Ex:loop2}
Let $Q$ be the loop from Example \ref{Ex:loop}. A short computer calculation shows that the multiplication group $\Mlt(Q)$ has size $24$, and the inner mapping group $\Inn(Q)$ consists of the permutations $()$, $(3,4)$, $(5,6)$, $(3,4)(5,6)$. In particular, $\Inn(Q)$ is a group of size $4$, hence an abelian group, and $Q$ is therefore an AIM loop.

$Q$ has only three subloops, namely the subsets $\{1\}$, $\{1,2\}$ and $\{1,2,3,4,5,6\}$. It so happens that all four nuclei, the commutant and the center of $Q$ are equal to the normal subloop $\{1,2\}$. We have $Z_0(Q)=1$, $Z_1(Q)=Z(Q)=\{1,2\}$ and $Z_2(Q) = \{1,2,3,4,5,6\}$ (because $Q/Z_1(Q)$ is an abelian group). Thus $\cl{Q}=2$.
\end{example}

Recall (or see \cite[Thm. 7.1]{rotman}) that if $Q$ is a group then
\begin{equation}\label{Eq:Iso}
    \text{$\Inn(Q)$ is isomorphic to $Q/Z(Q)$}.
\end{equation}
This result cannot be generalized to loops. For instance, we saw in Example \ref{Ex:loop2} that there is a loop $Q$ of size $6$ with $\Inn(Q)$ of size $4$ and $Q/Z(Q)$ of size $6/2=3$.

Now, if $Q$ is a group, we deduce from \eqref{Eq:Iso} that
\begin{equation}\label{Eq:ClassPlus1}
    \text{$\cl{\Inn(Q)}\le n$ if and only if $\cl{Q}\le n+1$}.
\end{equation}
Neither of the two implications in \eqref{Eq:ClassPlus1} generalizes to loops. Indeed, Vesanen found a loop $Q$ of size $18$ with $\cl{Q} = 3$ such that $\Inn(Q)$ is not even nilpotent, much less of nilpotency class at most $2$; see \cite{KP}. To falsify the other implication, we will see below that there exist loops with $\cl{\Inn(Q)}=n$ but with $\cl{Q}\ne n+1$. (Note, however, that Niemenmaa was able to prove recently that if $Q$ is finite and $\Inn(Q)$ is nilpotent then $Q$ is at least nilpotent \cite{Niem}.)

Using $n=1$ in \eqref{Eq:ClassPlus1}, we see that for a group $Q$
\begin{equation}
\label{Eq:AIMLike}
\begin{tabular}{l}
  $\Inn(Q)$ is abelian (that is, $Q$ is an AIM loop) \\
  if and only if $\cl{Q}\le 2$.
\end{tabular}
\end{equation}
Does this statement generalize to loops? The answer is of importance in loop theory because it sheds light on loops of nilpotency class $2$, in some sense the most immediate generalization of abelian groups to loops.

Bruck showed in \cite{Bruck-TAMS} that one implication of \eqref{Eq:AIMLike} holds for all loops: if $Q$ is a loop with $\cl{Q}\le 2$ then $Q$ is an AIM loop. For a long time, it was conjectured that the other implication of \eqref{Eq:AIMLike} also holds, and much work in loop theory was devoted to this problem \cite{CsK} \cite{Kepka} \cite{NK}. However, in 2004, \Csorgo \cite{Csorgo} disproved \eqref{Eq:AIMLike} by constructing an AIM loop $Q$ (of size $128$) with $\cl{Q}=3$.

Consequently, AIM loops $Q$ with $\cl{Q}\ge 3$ have come to be called \emph{loops of \Csorgo type}. Additional constructions of loops of \Csorgo type followed in rapid succession \cite{DrK} \cite{DV} \cite{GN}. We remark that it is still not known if there exists an AIM loop $Q$ with $\cl{Q}>3$.

What can be salvaged from the statement \eqref{Eq:AIMLike} for loops? Based on the structure of all known loops of \Csorgo type, the first-named author has been offering the following structural conjecture in various talks and presentations. This is the first published statement of the conjecture.

\begin{mainconj}
Let $Q$ be an AIM loop. Then $Q/N(Q)$ is an abelian group and $Q/Z(Q)$ is a group. In particular, $Q$ is nilpotent of class at most $3$.
\end{mainconj}

Three remarks are worth making here. First, the primary assertion of the Main Conjecture is actually somewhat stronger than the ``in particular'' part, that is, having nilpotency class $3$ does not necessarily imply $Q/N(Q)$ is abelian or $Q/Z(Q)$ is a group. Second, note that the Main Conjecture makes no reference to cardinality of the loop $Q$. It is certainly conceivable that the conjecture holds for all finite loops but that there is some infinite counterexample. Finally, it is tacit in the statement of the conjecture that the nucleus $N(Q)$ is a normal subloop of $Q$ (else the factor loop $Q/N(Q)$ cannot be formed). This is not true for arbitrary loops but is easy to show for AIM loops.

\bigskip

From the discussion so far, it may seem that the Main Conjecture is too high order to be attacked fruitfully by the methods of automated deduction. However, the hypotheses and conclusions of the conjecture can be given purely equationally as we now describe.

\bigskip

Bruck showed \cite{Bruck} that the inner mapping group is generated by three kinds of mappings that measure deviations from associativity and commutativity, namely,
\begin{displaymath}
    \Inn(Q) = \langle R_{x,y},\,T_x,\,L_{x,y};\;x,\,y\in Q\rangle,
\end{displaymath}
where
\begin{displaymath}
    R_{x,y} = R_x R_y R_{xy}\inv\,,
    \qquad
    T_x = R_x L_x\inv\,,
    \qquad
    L_{x,y} = L_x L_y L_{yx}\inv\,.
\end{displaymath}

Since a group is abelian if and only if any two of its generators commute, we immediately obtain the following characterization of AIM loops:

\begin{lemma}
\lemlabel{inn_eq}
A loop $Q$ is an AIM loop if and only if the following identities hold:
\begin{align*}
    T_x T_y &= T_x T_y,\\
    L_{x,y} L_{z,w} &= L_{z,w} L_{x,y},\\
    R_{x,y} R_{z,w} &= R_{z,w} R_{x,y},\\
    L_{x,y} T_z &= T_z L_{x,y}, \\
    R_{x,y} T_z &= T_z R_{x,y}, \\
    L_{x,y} R_{z,w} &= R_{z,w} L_{x,y}
\end{align*}
for all $x$, $y$, $z$, $w\in Q$.
\end{lemma}

To encode the conclusions of the Main Conjecture, it is useful to introduce two more functions, the \emph{associator}
\[
[x,y,z] = (x\cdot yz)\ldiv (xy\cdot z)
\]
and the \emph{commutator}
\[
[x,y]= (yx)\ldiv (xy)\,,
\]
that is, $(x\cdot yz)[x,y,z] = xy\cdot z$ and $yx\cdot [x,y] = xy$.
The former is a measure of nonassociativity and the latter is a measure of noncommutativity.
Different conventions are possible for each of these functions, \emph{e.g.} one could use
$(xy\cdot z)\rdiv (x\cdot yz)$ as a definition of associator. Our convention is the traditional one \cite{Bruck}.
Note that an element $a\in Q$ is in the left nucleus $N_{\lambda}(Q)$ if and only if $[a,x,y] = 1$ for all $x$, $y\in Q$, and that the other nuclei are similarly characterized.

The following is now a straightforward consequence of the definitions.

\begin{lemma}
\lemlabel{QNabelian}
Let $Q$ be a loop. Then
\begin{enumerate}
\item[(i)] If $N(Q)\unlhd Q$ then $Q/N(Q)$ is an abelian group if and only if the following identities hold:
\begin{align*}
[[x,y,z],u,v] &= [u,[x,y,z],v] = [u,v,[x,y,z]] = 1\,, \\
[[x,y],z,u] &= [z,[x,y],u] = [z,u,[x,y]] = 1
\end{align*}
for all $x,y,z,u,v\in Q$;
\item [(ii)] $Q/Z(Q)$ is a group if and only if the following identities hold:
\begin{align*}
[[x,y,z],u,v] = [u,[x,y,z],v] &= [u,v,[x,y,z]] = 1\,, \\
[[x,y,z],u] &= 1
\end{align*}
for all $x,y,z,u,v\in Q$.
\end{enumerate}
\end{lemma}
\begin{proof}
Let $S\unlhd Q$. The following conditions are equivalent: $(xS\cdot yS)\cdot zS = xS\cdot (yS\cdot zS)$, $(xy\cdot z)S = (x\cdot yz)S$, $((x\cdot yz)\ldiv (xy\cdot z))S = S$, $[x,y,z]S = S$, $[x,y,z]\in S$. Thus $Q/S$ is a group if and only if $[x,y,z]\in S$ for every $x$, $y$, $z\in Q$. Similarly, $Q/S$ is commutative if and only if $[x,y]\in S$ for every $x$, $y\in Q$.

The condition $[[x,y,z],u,v] = [u,[x,y,z],v] = [u,v,[x,y,z]] = 1$ says that $[x,y,z]\in N(Q)$ for all $x$, $y$, $z\in Q$, and the condition
$[[x,y],z,u] = [z,[x,y],u] = [z,u,[x,y]] = 1$ says that $[x,y]\in N(Q)$ for all $x$, $y\in Q$. This proves (i).

Concerning (ii), the condition $[[x,y,z],u] = 1$ says that $[x,y,z]\in C(Q)$ for every $x$, $y$, $z\in Q$. Since $Z(Q)=N(Q)\cap C(Q)$, we are done.\qed
\end{proof}

In Figure \ref{Fg:Input} we give a basic \textsc{Prover9} input file for the Main Conjecture. To encode loops, we use the definition with three binary operations $\cdot$, $\ldiv$ and $\rdiv$. The clauses labeled ``obvious compatibility'' are not necessary, but are included to help the search. Note how the assumptions on AIM loops correspond to the identities of Lemma \lemref{inn_eq}, and how the goals correspond to the identities of Lemma \lemref{QNabelian}.

\begin{figure}
\begin{small}
\begin{verbatim}
formulas(assumptions).
   % loop axioms
   1 * x = x.           x * 1 = x.
   x \ (x * y) = y.     x * (x \ y) = y.
   (x * y) / y = x.     (x / y) * y = x.

   % associator
   (x * (y * z)) \ ((x * y) * z) = a(x,y,z).

   % commutator
   (x * y) \ (y * x) = K(y,x).

   % inner mappings
   % L(u,x,y) = u L(x) L(y) L(yx)^{-1}
   (y * x) \ (y * (x * u)) = L(u,x,y).

   % R(u,x,y) = u R(x) R(y) R(xy)^{-1}
   ((u * x) * y) / (x * y) = R(u,x,y).

   % T(u,x) = u R(x) L(x)^{-1}
   x \ (u * x) = T(u,x).

   % obvious compatibility
   a(x,y,z) = 1 -> L(z,y,x) = z.    L(x,y,z) = x -> a(z,y,x) = 1.
   T(x,y) = x -> T(y,x) = y.        T(x,y) = x -> K(x,y) = 1.
   K(x,y) = 1 -> T(x,y) = x.

   % abelian inner mapping group (AIM loop)
   T(T(u,x),y) = T(T(u,y),x).
   L(L(u,x,y),z,w) = L(L(u,z,w),x,y).
   R(R(u,x,y),z,w) = R(R(u,z,w),x,y).
   T(L(u,x,y),z) = L(T(u,z),x,y).
   T(R(u,x,y),z) = R(T(u,z),x,y).
   L(R(u,x,y),z,w) = R(L(u,z,w),x,y).
end_of_list.

formulas(goals).
   a(K(x,y),z,u) = 1              # label("aK1").
   a(x,K(y,z),u) = 1              # label("aK2").
   a(x,y,K(z,u)) = 1              # label("aK3").
   K(a(x,y,z),u) = 1              # label("Ka").
   a(a(x,y,z),u,w) = 1            # label("aa1").
   a(x,a(y,z,u),w) = 1            # label("aa2").
   a(x,y,a(z,u,w)) = 1            # label("aa3").
end_of_list.
\end{verbatim}
\end{small}
\caption{A \textsc{Prover9} input file for the problem ``In an AIM loop $Q$, is $Q/N(Q)$ an abelian group
and is $Q/Z(Q)$ a group?''}
\label{Fg:Input}
\end{figure}

A resolution one way or the other of the full Main Conjecture would be a major milestone in loop theory, alas, the answer remains elusive. Nevertheless we managed to confirm the Main Conjecture for several classes of loops. We describe one of the successful cases in {\S}\secref{lcloops} and list others in {\S}\secref{finale}.

\section{The Strategy}
\seclabel{strategy}

Our search for proofs for the various AIM cases involves sequences of
\textsc{Prover9} experiments that rely heavily on the use of
hints~\cite{Hints} and on the method of proof sketches~\cite{Sketches}.
Under the hints strategy, a generated clause is given special consideration if
it \emph{matches} (subsumes) a user-supplied hint clause.  In \textsc{Prover9},
the actions associated with hint matching are controllable by the user, but
the most typical action is to give hint matchers the highest
priority in the proof search.

A \emph{proof sketch} for a theorem $T$ is a sequence of clauses giving a
set of conditions \emph{sufficient} to prove $T$.  In the ideal case, a
proof sketch consists of a sequence of lemmas, where each lemma is fairly
easy to prove.  In any case, the clauses of a proof sketch identify
potentially notable milestones on the way to finding a proof.
From a strategic standpoint, it is desirable to recognize when we have
achieved such milestones and to adapt the continued search for a proof
accordingly.  In particular, we want to focus our attention
on such milestone results and pursue their consequences sooner rather
than later.  The hints mechanism provides a natural and effective way
to take full advantage of proof sketches in the search for a proof.

The use of hints is additive in the sense that hints from multiple
proof sketches or from sketches for different parts of a proof can
all be included at the same time.  For this reason, hints are
particularly valuable for ``gluing'' subproofs together and completing
partial proofs, even when wildly different search strategies were used
to find the individual subproofs.

In \cite{Sketches}, we consider how the generation and use of proof sketches,
together with the sophisticated strategies and procedures supported by
an automated reasoning program such as \textsc{Prover9}, can be used to
find proofs to challenging theorems, including open questions.
The general approach is to find proofs with additional assumptions
and then to systematically eliminate these assumptions from the input set,
using all previous proofs as hints.  It also can be very effective to
include as proof sketches proofs of related theorems in the same area
of study.  In the AIM study, for example, proofs for the LC case (see Section \secref{lcloops}) were
found by first proving the LC case with the additional assumption
``left inner maps preserve inverses'',

\begin{verbatim}
   L(x,y,z) \ 1 = L(x \ 1,y,z).
\end{verbatim}

\noindent
The resulting proofs, together with proofs of other previously proved cases,
were included as proof sketches in the search that found the proofs for
LC alone.

Our strategy for searching for AIM proofs is based on the following
two observations.
\begin{itemize}
\item Proofs for the various special cases tend to share several key steps.
       This makes these problems especially amenable to the use of previously
       found proofs as hints.
\item Proof searches in this problem area tend to be especially sensitive
       to the underlying lexical ordering of terms.  This is due primarily
       to the resulting effect on the demodulation (rewriting) of deduced
       clauses.
\end{itemize}

Addressing the second observation, we can run multiple searches, trying each
of several term orderings in turn.  But rather than simply running each of
these searches as independent attempts, we can leverage \textsc{Prover9}'s
hints mechanism to take advantage of any apparent progress made in previous
searches.  In particular, after running a search with one term ordering,
we can gather the derived clauses that match hints and include these as
input assumptions for the following runs.

The second-named author has automated this approach with a utility called
\texttt{p9loop}. \texttt{P9loop} takes
as input an ordinary \textsc{Prover9} input file (including hints)
and a list of term ordering directives and proceeds as follows.

\begin{enumerate}
\item[$\bullet$] Run \textsc{Prover9} with a term ordering from the list until either a
       proof is found or a user-specified processing limit is reached.
\item[$\bullet$] If no proof has been found, restart \textsc{Prover9} with the next term
       ordering from the list, including all of the previous hint matchers
       as additional input assumptions.
\end{enumerate}

We have had numerous successes using this approach, sometimes finding proofs
that rely on several \texttt{p9loop} iterations.  We have, for
example, found proofs after iterating through over 50 term orderings,
with as many as 30 of the iterations contributing to the found proof.

We note that the final iteration of a successful \texttt{p9loop} execution generally
results in a proof that includes assumptions derived in previous \textsc{Prover9}
runs and that we do not immediately have the derivations of these assumptions.
Furthermore, this property may be nested in that an assumption
from a previous run may in turn rely on assumptions
from even earlier runs.  In order to get a complete proof of the final
theorem, we use \textsc{Prover9} to recover the missing derivations by systematically
eliminating these assumptions in a way that is analogous to the general proof sketches method.

\section{LC Loops}
\seclabel{lcloops}

A loop $Q$ is said to be an \emph{LC loop} if it satisfies any of the following equivalent identities:
\begin{align*}
x(x(yz)) &= (x(xy))z, \\
x(x(yz)) &= ((xx)y)z, \\
(xx)(yz) &= (x(xy))z, \\
x(y(yz)) &= (x(yy))z
\end{align*}
for all $x,y,z\in Q$. LC loops were introduced by Fenyves \cite{Fenyves} as one of the loops of \emph{Bol-Moufang type}. They were studied in more detail in \cite{PV-BM1} where the equivalence of the identities above was proven.

We have been able to establish the Main Conjecture in the special case of LC loops, that is, we have the following theorem.

\begin{theorem}
\thmlabel{LCthm}
Let $Q$ be an AIM LC loop. Then $Q/N(Q)$ is an abelian group and $Q/Z(Q)$ is a group. In particular, $Q$ is nilpotent of class at most $3$.
\end{theorem}

It sometimes is desirable to obtain a humanized proof of a result first found by means of automated deduction tools. Generally speaking, the process of ``humanization'' can give some higher-level insights into the structure being studied and can often result in a simpler (from a human perspective) proof than the one originally generated by an automated deduction tool. Many of the references in the bibliography feature humanized proofs.

There are various varieties of loops with abelian inner mappings for which the proof
of the conjecture will be worth humanizing. We list a few of these in the next
section. However, that desire for a human proof is tempered by the law of diminishing returns. In some cases, a human proof may not be worth the time and effort put into
translating the automated proof.  For LC loops, we run right
into this issue.

We often prefer to produce automated proofs that are strictly-forward
derivations of the theorems and that are free of any applications of
demodulation.  This sometimes is for purposes of presentation, but most
often it's because we find that these proofs make better proof sketches
for future searches, especially in a larger study such as the AIM project.
In \textsc{Prover9}, these constraints can be satisfied with the
directives \texttt{set(restrict\_denials)} and \texttt{clear(back\_demod)}
respectively.  If the input list \texttt{formulas(demodulators)} is empty,
including the directive \texttt{clear(back\_demod)} ensures that no
clause---input or derived---will be used as a demodulator.

Here is the data for the strictly-forward, demodulation-free proofs
of the seven LC goals:

\begin{center}
\begin{tabular}{ccc}
Goal & Length & Level \\
\hline
Ka & 2192 & 222 \\
aK1 & 2191 & 221 \\
aK2 & 2199 & 224 \\
aK3 & 2394 & 276 \\
aa1 & 2842 & 321 \\
aa2 & 2842 & 321 \\
aa3 & 2841 & 320
\end{tabular}
\end{center}

\noindent Of course, the seven proofs have many clauses in common and some proofs are virtually identical. For example, it is known that in LC loops, the left and middle nuclei coincide. The goals aa1 and aa2 state that associators are in those nuclei, so it is not surprising that their proof lengths and levels are the same.
The proofs (and corresponding input file) can be found on this paper's
associated Web page~\cite{AIM-LC-Web}.

Since at present, the authors think that the classes of loops discussed in the next section will be more worthy of human translation, any attempt to do so for LC loops has been postponed.

\section{Further Remarks}
\seclabel{finale}

Within certain varieties of loops, it is known that there are no loops of \Csorgo type. In other words, given an AIM loop $Q$ in that variety, $Q$ is necessarily nilpotent of class at most $2$.

In order to formulate the problem of showing that an AIM loop in a particular variety is centrally nilpotent of class at most $2$, it is only necessary to add the defining equations of the variety to the assumptions and to add one more goal:
\[
[[x,y],z] = 1
\]
for all $x,y,z$. Indeed, we already know from the identities of Lemma \lemref{QNabelian} that $Q/Z(Q)$ is a group, that $[x,y]\in N(Q)$, and the last goal says that $[x,y]\in C(Q)$, so $[x,y]\in Z(Q)=C(Q)\cap N(Q)$ and $Q/Z(Q)$ is an abelian group, i.e., $\cl{Q}\le 2$.

\begin{figure}
\begin{scriptsize}
\input{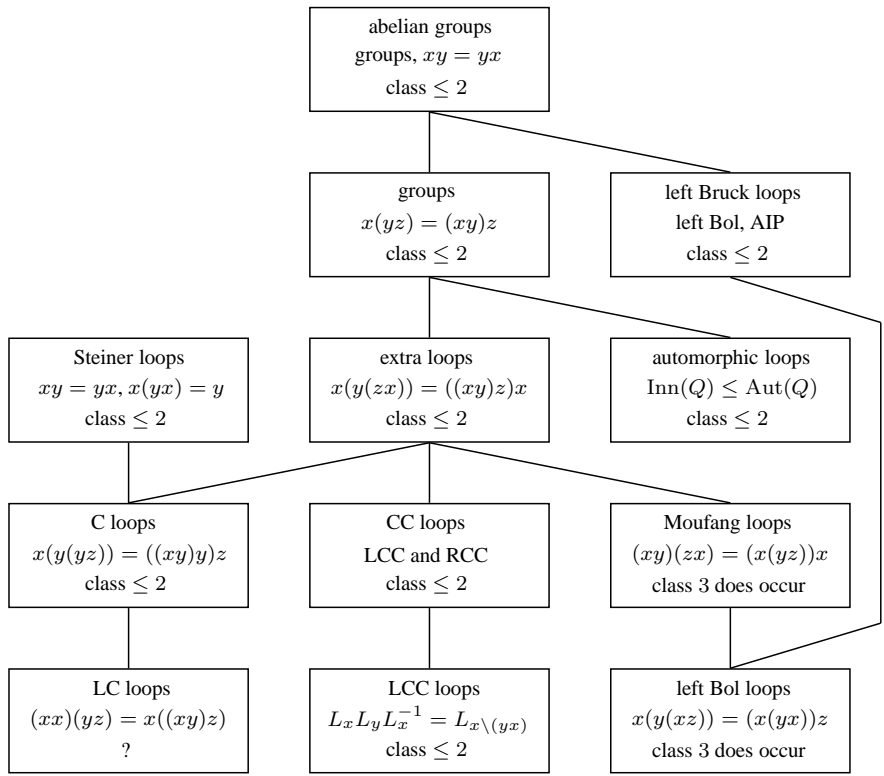}
\end{scriptsize}
\caption{Well-studied varieties of loops in which the Main Conjecture is true.}
\label{Fg:Chart}
\end{figure}

Figure \ref{Fg:Chart} summarizes what is known about the Main Conjecture and about the above-mentioned problem in several well-studied varieties of loops. The varieties include abelian groups, groups, \emph{Steiner loops} (defined by the identities $xy=yx$, $x(yx)=y$), \emph{extra loops} (defined by $x(y(zx)) = ((xy)z)x$), \emph{automorphic loops} (loops where all inner mappings are automorphisms), \emph{C loops} (defined by $((xy)y)z = x(y(yz))$), \emph{conjugacy closed} (or \emph{CC}) \emph{loops} (loops in which $L_xL_yL_x^{-1}$ is a left translation and $R_xR_yR_x^{-1}$ is a right translation for every $x$, $y$), LC loops, \emph{left Bol loops} (defined by $x(y(xz)) = (x(yx))z$), \emph{left conjugacy closed} (or \emph{LCC}) \emph{loops} (loops in which $L_xL_yL_x^{-1}$ is a left translation for every $x$, $y$), and \emph{left Bruck loops} (that is, left Bol loops with the \emph{automorphic inverse property} (or \emph{AIP}) $(xy)^{-1}=x^{-1}y^{-1}$). To keep the figure legible, we omitted all dual varieties (RC loops, right Bol loops, RCC loops, right Bruck loops) for which analogous results hold.

The varieties in Figure \ref{Fg:Chart} are listed with respect to inclusion, with smaller varieties higher up. For instance, every Moufang loop is a left Bol loop.

All varieties in Figure \ref{Fg:Chart} satisfy the Main Conjecture. Moreover, some varieties have the stronger property discussed above that an AIM loop $Q$ satisfies $\cl{Q}\le 2$. This is indicated by ``class $\le 2$'' in the figure. The already known cases where $\cl{Q}\leq 2$ are AIM LCC loops \cite{CD} and AIM left Bruck loops \cite{PSt2}.

All other cases indicated in the figure---automorphic loops, Moufang loops, left Bol loops and C loops---comprise new results which are part of this project. Their proofs will appear elsewhere, often in humanized form. Another new result is that an AIM Moufang loop $Q$ satisfies $\cl{Q}\le 2$ if $Q$ is \emph{uniquely $2$-divisible}, that is, the mapping $x\mapsto x^2$ is a bijection of $Q$. (Previously it was known that an AIM Moufang loop that is uniquely $2$-divisible and \emph{finite} satisfies $\cl{Q}\le 2$, and that there is an AIM Moufang loop $Q$ of size $2^{14}$ satisfying $\cl{Q}=3$ \cite{GN}.) The same result holds for uniquely $2$-divisible AIM left Bol loops.

Despite some serious effort, we have not been able to decide if an AIM LC loop $Q$ satisfies $\cl{Q}\le 2$ (indicated in the figure by ``?''); we believe that it does.

Finally, all other claims contained within the figure are consequences of the inclusions.

\bibliographystyle{plain}

\begin{thebibliography}{99}

\bibitem{Belousov} Belousov, V. D.:
Foundations of the Theory of Quasigroups and Loops.
(in Russian) Izdat. Nauka, Moscow (1967).

\bibitem{Bruck-TAMS} Bruck, R. H.:
Contributions to the theory of loops.
Trans. Amer. Math. Soc. 60, 245--354 (1946).

\bibitem{Bruck} Bruck, R. H.:
A Survey of Binary Systems.
Springer-Verlag, Berlin-G\"{o}ttingen-Heidelberg (1971).

\bibitem{Csorgo} Cs\"{o}rg\H{o}, P.:
Abelian inner mappings and nilpotency class greater than two.
European J. Combin. 28, 858--867 (2007).

\bibitem{CD}  Cs\"{o}rg\H{o}, P. and Dr\'{a}pal, A.:
Left conjugacy closed loops of nilpotency class two.
Results Math. 47, 242--265 (2005).

\bibitem{CsK} Cs\"{o}rg\H{o}, P. and Kepka, T.:
On loops whose inner permutations commute.
Comment. Math. Univ. Carolin. 45, 213--221 (2004).

\bibitem{DrK} Dr\'{a}pal, A. and Kinyon, M. K.:
Buchsteiner loops: associators and constructions.
\url{arXiv/0812.0412}.

\bibitem{DV} Dr\'{a}pal, A. and Vojt\v{e}chovsk\'{y}, P.:
Explicit constructions of loops with commuting inner mappings.
European J. Combin. 29, 1662--1681 (2008).

\bibitem{Evans} Evans, T.:
Homomorphisms of non-associative systems.
J. London Math. Soc. 24, 254--260 (1949).

\bibitem{Fenyves} Fenyves, F.:
Extra loops II. On loops with identities of Bol-Moufang type.
Publ. Math. Debrecen, 16, 187--192 (1969).

\bibitem{GAP} The GAP Group:
GAP -- Groups, Algorithms, and Programming, Version 4.4.10; 2007.
\url{www.gap-system.org}

\bibitem{Hullot} Hullot, J.-M.:
A catalogue of canonical term rewriting systems.
Technical Report CSC 113, SRI International (1980).

\bibitem{Kepka} Kepka, T.:
On the abelian inner permutation groups of loops.
Comm. Algebra 26, 857--861 (1998).

\bibitem{KP} Kepka T. and Phillips, J. D.:
Connected transversals to subnormal subgroups.
Comment. Math. Univ. Carolin. 38, 223--230 (1997).

\bibitem{KepKinPhil1} Kepka, T., Kinyon, M. K. and Phillips, J. D.:
The structure of F-quasigroups.
J. Algebra 317, 435--461 (2007).

\bibitem{diassoc} Kinyon, M. K., Kunen, K. and Phillips, J. D.:
Every diassociative A-loop is Moufang.
Proc. Amer. Math. Soc. 130, 619--624 (2002).

\bibitem{general} Kinyon, M. K., Kunen, K., and Phillips,  J. D.:
A generalization of Moufang and Steiner loops.
Algebra Universalis 48, 81--101 (2002).

\bibitem{dicc} Kinyon, M. K., Kunen, K. and Phillips, J. D.:
Diassociativity in conjugacy closed loops.
Comm. Algebra 32, 767--786 (2004).

\bibitem{extra} Kinyon, M. K. and Kunen, K.:
The structure of extra loops.
Quasigroups and Related Systems 12, 39--60 (2004).

\bibitem{pacc} Kinyon, M. K. and Kunen, K.:
Power-associative, conjugacy closed loops.
J. Algebra 304, 679--711 (2006).

\bibitem{Cloops2} Kinyon,  M. K., Phillips, J. D. and Vojt\v{e}chovsk\'{y},  P.:
C-loops: extensions and constructions.
J. Algebra Appl. 6, 1--20 (2007).

\bibitem{commutant} Kinyon,  M. K., Phillips, J. D. and Vojt\v{e}chovsk\'{y}, P.:
When is the commutant of a Bol loop a subloop?
Trans. Amer. Math. Soc. 360, 2393--2408 (2008).

\bibitem{AIM-LC-Web} Kinyon,  M. K., Veroff, R. and Vojt\v{e}chovsk\'{y}, P.:
Loops with abelian inner mapping groups: an application of automated deduction
(Web support). (2012) \url{www.cs.unm.edu/~veroff/AIM_LC/}.

\bibitem{KB} Knuth, D. E. and Bendix, P. B.:
Simple word problems in universal algebras.
In J. Leech (ed.), Proceedings of the Conference on Computational
Problems in Abstract Algebras, Oxford, 1967, Pergamon Press, Oxford, 263--298 (1970).

\bibitem{moufang} Kunen, K.:
Moufang quasigroups.
J. Algebra 183, 231--234 (1996).

\bibitem{laws} Kunen, K.:
Quasigroups, loops, and associative laws.
J. Algebra 185, 194--204 (1996).

\bibitem{rings} Kunen, K.:
Alternative loop rings.
Comm. Algebra 26, 557--564 (1998).

\bibitem{gloops} Kunen, K.:
G-Loops and permutation groups.
J. Algebra 220, 694--708 (1999).

\bibitem{cc} Kunen, K.:
The structure of conjugacy closed loops.
Trans. Amer. Math. Soc. 352, 2889--2911 (2000).

\bibitem{Otter} McCune,  W.:
Otter 3.0 Reference Manual and Guide.
Tech. Report ANL-94/6, Argonne National Laboratory, Argonne, IL (1994).
See also \url{www.mcs.anl.gov/AR/otter/}

\bibitem{Prover9} McCune, W.:
Prover9, version 2009-02A, \url{www.cs.unm.edu/~mccune/prover9/}

\bibitem{LOOPS} Nagy, G. and Vojt\v{e}chovsk\'{y}, P.:
LOOPS: Computing with quasigroups and loops in GAP
-- a GAP package, version 2.0.0, (2008). \url{www.math.du.edu/loops}

\bibitem{GN} Nagy, G. P. and Vojt\v{e}chovsk\'{y}, P.:
Moufang loops with commuting inner mappings.
J. Pure Appl. Algebra 213, 2177--2186 (2009).

\bibitem{Niem} Niemenmaa, M.:
Finite loops with nilpotent inner mapping groups are
centrally nilpotent.
Bull. Aust. Math. Soc. 79, 109--114 (2009).

\bibitem{NK} Niemenmaa, M. and Kepka, T.:
On connected transversals to abelian subgroups in finite groups.
Bull. London Math. Soc. 24, 343--346 (1992).

\bibitem{P-short} Phillips, J. D.:
A short basis for the variety of WIP PACC-loops.
Quasigroups Related Systems 14, 73--80 (2006).

\bibitem{P-Moufang} Phillips, J. D.:
The Moufang laws, global and local.
J. Algebra Appl. 8, 477--492 (2009).

\bibitem{PSh} Phillips, J. D. and Shcherbacov, V. A.:
Cheban loops.
J. Gen. Lie Theory Appl. 4 (2010), Art. ID G100501, 5 pp.

\bibitem{PSt1} Phillips, J. D. and Stanovsk\'{y}, D.:
Automated theorem proving in quasigroup and loop theory.
AI Commun. 23, 267--283 (2010).

\bibitem{PSt2} Phillips, J. D. and Stanovsk\'{y}, D.:
Bruck loops with abelian inner mapping groups.
Comm. Alg., to appear.

\bibitem{PV-BM1} Phillips, J. D. and Vojt\v{e}chovsk\'{y}, P.:
The varieties of loops of Bol-Moufang type.
Algebra Universalis 54, 259--271 (2005).

\bibitem{PV-BM2} Phillips, J. D. and Vojt\v{e}chovsk\'{y}, P.:
The varieties of quasigroups of Bol-Moufang type: an equational reasoning approach.
J. Algebra 293, 17--33 (2005).

\bibitem{PV-C} Phillips, J. D. and Vojt\v{e}chovsk\'{y}, P.:
C-loops: an introduction.
Publ. Math. Debrecen 68, 115--137 (2006).

\bibitem{PV-Scoop} Phillips, J. D. and Vojt\v{e}chovsk\'{y}, P.:
A scoop from groups: equational foundations for loops.
Comment. Math. Univ. Carolin. 49, 279--290 (2008).

\bibitem{Pflugfelder} Pflugfelder, H. O.:
Quasigroups and Loops: Introduction.
Sigma Series in Pure Math. 8, Heldermann Verlag, Berlin (1990).

\bibitem{rotman} Rotman, J. J.:
An Introduction to the Theory of Groups.
4th edition. Springer-Verlag, New York (1995).

\bibitem{Hints} Veroff, R.:
Using hints to increase the effectiveness of an automated reasoning
program:  case studies.
J. Automated Reasoning 16, 223--239 (1996).

\bibitem{Sketches} Veroff, R.:
Solving open questions and other challenge problems using proof sketches.
J. Automated Reasoning 27, 157--174 (2001).

\bibitem{PSATO} Zhang, H., Bonacina, M. P. and Hsiang, J.:
PSATO: a distributed propositional prover and its application to quasigroup problems.
J. Symbolic Computation 21, 543--560 (1996).

\end{thebibliography}

\end{document}